\documentclass[twoside, 12pt, reqno]{amsart}

\DeclareSymbolFont{symbols}{OMS}{cmsy}{m}{n}

\usepackage{fullpage}
\usepackage[utf8]{inputenc}
\usepackage[T1]{fontenc}

\usepackage{mathrsfs}
\usepackage{mathtools}
\mathtoolsset{showonlyrefs,showmanualtags}
\usepackage{amsfonts}
\usepackage{amsmath, amssymb, amsthm}
\usepackage{enumerate}
\usepackage{graphicx, tikz}
\usepackage{ifthen}
\usetikzlibrary{arrows, scopes}
\usetikzlibrary{calc, cd}
\usetikzlibrary{decorations.pathreplacing}
\usetikzlibrary{decorations.pathmorphing}
\usetikzlibrary{patterns}
\usetikzlibrary{fadings}
\usetikzlibrary{cd}
\usetikzlibrary{automata, positioning}
\usetikzlibrary{calc}
\usepackage{tikz-cd}
\usepackage{stmaryrd}
\usepackage{setspace}
\usepackage{hyperref}
\usepackage[alphabetic, initials]{amsrefs}
\usepackage{bbm}

\newtheorem{thm}{Theorem}[section]
\newtheorem{lem}[thm]{Lemma}
\newtheorem{cor}[thm]{Corollary}
\newtheorem{prop}[thm]{Proposition}
\theoremstyle{definition}
\newtheorem{defi}[thm]{Definition}

\theoremstyle{remark}
\newtheorem{rmk}[thm]{Remark}

\newcommand{\Hil}{\mathcal{H}}
\newcommand{\slot}{{~\cdot~}}

\newcommand{\Spin}{\mathop{\mathsf{Spin}}}
\newcommand{\Sp}{\mathop{\mathsf{Sp}}}

\newcommand{\SU}{\mathop{\mathsf{SU}}}

\newcommand{\U}{{\mathsf{U}}}

\newcommand{\A}{\mathcal{A}}
\newcommand{\cI}{\mathcal{I}}

\newcommand{\RR}{\mathbb{R}}
\newcommand{\TT}{\mathbb{T}}

\newcommand{\ZZ}{\mathbb{Z}}
\newcommand{\NN}{\mathbb{N}}

\DeclareMathOperator{\Diff}{Diff}

\DeclareMathOperator{\Hom}{Hom}

\DeclareMathOperator{\Ad}{Ad}

\DeclareMathOperator{\id}{id}
\DeclareMathOperator{\Tr}{Tr}

\DeclareMathOperator{\B}{B}
\DeclareMathOperator{\supp}{supp}

\DeclareMathOperator{\Aut}{Aut}

\usepackage{enumerate}  
\usepackage{enumitem}
\newcommand{\e}{\mathrm{e}}
\newcommand{\dd}{\mathrm{d}}
\newcommand{\ima}{\mathrm{i}}

\newcommand{\Lpc}{\mathcal{L}}

\newcommand{\op}{\mathrm{op}}

\usepackage{xspace}

\DeclareRobustCommand{\ie}{i.e.\@\xspace}
\makeatletter
\DeclareRobustCommand{\etc}{%
    \@ifnextchar{.}%
        {etc}%
        {etc.\@\xspace}%
}
\makeatother

\newcommand{\threeone}{III$_1$\@\xspace}

\def\u1net{{\A_\RR}}

\mathtoolsset{showonlyrefs,showmanualtags}

\newcommand{\LL}{\mathrm{L}}
\newcommand{\basicform}[2]{\langle #1,#2\rangle_{\mathrm{basic}}}

\DeclareMathOperator{\Out}{\mathsf{Out}}

\DeclareMathOperator{\AUT}{\mathsf{Aut}}

\newcommand{\SOne}{\mathrm{S}^1}

\usepackage{tikz-cd}

\newcommand{\Borel}{\mathrm{Borel}}

\newcommand{\LieG}{\mathfrak{g}}
\newcommand{\LieT}{\mathfrak{t}}
\newcommand{\rootsys}{\Phi}

\newcommand{\Qvee}{Q^{\vee}}
\newcommand{\Rvee}{\rootsys^{\vee}}
\newcommand{\weight}{P}

\DeclareMathOperator{\Weyl}{W}

\DeclareMathOperator{\Sym}{Sym}

\begin{document}
\date{\today}
\dateposted{\today}
\newcommand{\mytitle}{%
Anomalies for compact simple Lie groups via $G$-kernels for loop group nets}
\title{\mytitle}
\author{Marcel Bischoff}
\email{mrclbschff@gmail.com}
\author{Pradyut Karmakar}
\address{Sam Houston State University, 332 A LDB, 1900 Avenue I, Huntsville, TX 77340}
\email{karmakar.pradyut@gmail.com}
\dedicatory{}

\begin{abstract}
  Let $G$ be a compact simply connected simple Lie group and $k\ge1$.
  Jump loops give an outer homomorphism $\theta_k$ from $G$ to the outer
  automorphisms of the local factor of the level-$k$ loop group net,
  which is the hyperfinite type \threeone factor. We show that its
  Sutherland obstruction is $k$ times the positive basic class in
  $H^3_{\Borel}(G,\TT)\cong H^4(BG,\ZZ)$ and that $\theta_k$ is
  faithful, so it is a $G$-kernel. Central elements act by the outer
  automorphisms of the simple currents. In types $A$, $D$, $E$ the level-one class
  is computed by comparing maximal-torus jump loops with the lattice
  $\alpha$-induction construction of the companion paper~\cite{BiKa} on
  both current and charged generators, followed by integral degree-four
  Cartan restriction. Higher levels follow by tensor additivity, and the
  remaining types by restriction along index-one embeddings into simply
  laced groups. As a consequence, for every compact simply connected
  semisimple $G$, every class in $H^4(BG,\ZZ)$ is the obstruction of a
  $G$-kernel on the hyperfinite type \threeone factor.
\end{abstract}

\maketitle

\setcounter{tocdepth}{3}
\tableofcontents

\section{Introduction}
\label{sec:Introduction}

The program of realizing Sutherland's obstruction~\cites{Su1980, Su1980-II} for
continuous compact groups as kernels on hyperfinite type \threeone factors was
initiated in the companion paper~\cite{BiKa}. There we considered the
abelian case: for an even lattice $L\subseteq F$ inside a finite-dimensional
Euclidean space, the conformal net $\A_L$ of~\cite{DoXu2006} associated with $L$
admits an $F/L$-kernel obtained from $\alpha$-induction through the local
extension $\A_F\subseteq \A_L$, and its Sutherland obstruction is the even form
of $L$ seen as a class in $H^{4}(B(F/L),\ZZ)\cong \operatorname{Sym}^{2}(L^*)$.
Here ``$F/L$-kernel'' means a monomorphism into the \emph{outer} automorphism
group. The lattice extension also has the honest dual gauge action of
$F/L^*$, fixing $\A_F$ pointwise. The outer kernel is represented by
twisted sectors whose boundary action is the image under $F/L\to F/L^*$;
see Section~\ref{ssec:GaugeVsLocal}.
As a corollary, every obstruction for the $n$-torus $\U(1)^{n}$ is realized on
  the hyperfinite type \threeone factor (\cite{BiKa}, the corollary on
  realization of all torus obstructions in the cohomology computation).

The main result of this note is the following. Write $R$ for
the hyperfinite type \threeone factor and $c_G\in H^4(BG,\ZZ)$ for the
positive basic class, normalized by its coroot quadratic form, and
\[
  \tau\colon H^4(BG,\ZZ)\longrightarrow H^3_{\Borel}(G,\TT)
\]
for the characteristic-class identification.

\begin{thm}[Anomaly is the level]
  \label{thm:MainIntro}
  Let $G$ be a compact connected simply connected simple Lie group and
  $k\ge1$. Jump loops define a $G$-kernel
  \[
    \theta_k\colon G\longrightarrow\Out(\A_{G_k}(I))
  \]
  on the local factor $\A_{G_k}(I)\cong R$ of the level-$k$ loop group
  net, and its Sutherland obstruction is $\tau(kc_G)$. Consequently, for
  every compact simply connected semisimple $G$, every class in
  $H^4(BG,\ZZ)$ is the obstruction of a $G$-kernel on $R$.
\end{thm}

This is Theorem~\ref{thm:AnomalyAllTypes} together with
Corollary~\ref{cor:EveryObstruction} and
Remark~\ref{rmk:EveryObstructionSemisimple}. The input is the simply laced
level-one case, where the loop group net is a lattice net:

\begin{thm}[Simply laced realization]
  \label{thm:SimplyLacedRealization}
  Let $G$ be a compact connected simply connected simple Lie group of type
  $A_n$, $D_n$ ($n\ge4$), $E_6$, $E_7$, or $E_8$. For every integer
  $k\ge1$ there is a $G$-kernel
  \[
    \theta^{[k]}\colon G\longrightarrow\Out(R)
  \]
  whose Sutherland obstruction is $\tau(kc_G)$.
  More generally, for $G=\prod_jG_j$ with simply laced simple factors,
  every tuple of positive levels $(k_j)_j$ is realized in this way.
\end{thm}

At level one the kernel acts on $\A_{G_1}(I)$, using its root-lattice
realization; for $k>1$ the theorem uses $\A_{G_1}(I)^{\bar\otimes k}\cong R$.
The rank-one level-one construction is the seed case:

\begin{thm}[The level-one $\SU(2)$-kernel]
  \label{thm:SU2main}
  Let $\A_L\cong \A_{\SU(2)_1}$ be the conformal net associated with the
  root lattice $L=A_1=\sqrt2\,\ZZ$ and let $I\in\cI$. The jump-loop
  construction gives a Borel outer homomorphism
  \[
    \theta\colon \SU(2)\longrightarrow \Out(\A_L(I))
  \]
  admitting a Borel automorphism lift. Its maximal-torus restriction is
  the $T_L$-kernel of~\cite{BiKa}. Consequently $\theta$ is an
  $\SU(2)$-kernel and its Sutherland obstruction is the
  positive basic generator under
  \[
    H^3_{\Borel}(\SU(2),\TT)\cong H^4(B\SU(2),\ZZ)\cong\ZZ.
  \]
\end{thm}

The proof of Theorem~\ref{thm:MainIntro} has three steps.

\smallskip

\noindent\textbf{(i) Loop group realization.} The outer homomorphisms are
implemented by the level-$k$ vacuum representation $\pi_k$ of the central
extension $\Lpc G_k$ on the local von Neumann algebra
$\A_{G_k}(I)=\pi_k(\Lpc_I G)''$~\cites{PS1986, Wa, HenColimits}; their
composition defects can be localized in $I$ after separating the return
loop in $I'$. In every type the loop group net is completely rational and
its local algebras are hyperfinite of type \threeone
(Corollary~\ref{cor:LoopGroupNetAnalytic}).

\smallskip

\noindent\textbf{(ii) Torus comparison in types $A$, $D$, $E$.} At level
one the nonabelian local algebra is the lattice factor. There we identify
the torus jump automorphisms with the induced lifts of~\cite{BiKa} on both
current and charged generators. Torus faithfulness then implies group
faithfulness, and integral degree-four Cartan
restriction~\cite{Hen17}*{Theorem~6} identifies the obstruction with the
level. This is the precise role of the $A_1$ computation of~\cite{BiKa};
the known level of the vacuum representation by itself does not compute
the Sutherland obstruction.

\smallskip

\noindent\textbf{(iii) Levels and types.} Realizing the level-$k$ vacuum
inside the $k$-fold tensor power of the level-one vacuum multiplies the
associator by $k$ (Proposition~\ref{prop:DiagonalAdditivity}). The same
operator restriction along an index-one embedding into a simply laced
group gives the class in types $B_n$, $C_n$, $F_4$, $G_2$
(Proposition~\ref{prop:IndexOneRestriction}). Faithfulness reduces to the
centre, where central elements act by the outer automorphisms of the
simple currents (Proposition~\ref{prop:CentreOuter}).

\smallskip

Outside types $A$, $D$, $E$, the lattice route would need a comparison on
the ambient factor, which Theorem~\ref{thm:main} takes as a hypothesis. We
do not supply it; step (iii) bypasses it.

\subsection*{Organization}
Section~\ref{sec:Notation} fixes notation and recalls the loop-group conformal
nets and Sutherland obstruction. 
Section~\ref{sec:kernels} constructs the outer homomorphism from jump loops.
Section~\ref{sec:SU2} proves the torus-jump identification and
Theorem~\ref{thm:SU2main}. Section~\ref{sec:General} proves
Theorem~\ref{thm:SimplyLacedRealization} by tensoring, states the
conditional torus-route identification in Theorem~\ref{thm:main},
realizes every class (Corollary~\ref{cor:EveryObstruction}), and proves
the all-type identification Theorem~\ref{thm:AnomalyAllTypes}.
Section~\ref{sec:AnomalyIsLevel} records the resulting identification
\emph{anomaly $=$ level} as the universal phenomenon behind the computation.

\subsection*{Acknowledgements}
We thank James Tener for discussions on loop groups.

\subsection*{AI statement}
Claude AI has been used for fixing writing, grammar, typos and searching references. All mathematical ideas and proofs are solely due to authors.
\section{Notation and preliminaries}
\label{sec:Notation}

Unless a product is explicitly specified, $G$ is a compact, connected,
simply connected \emph{simple} Lie group. Its real Lie algebra is
$\LieG=\mathrm{Lie}(G)$. Fix a maximal torus $T\subset G$, with real
Cartan Lie algebra $F=\LieT=\mathrm{Lie}(T)$. Write
$\rootsys\subset \LieT^{*}$ for the root system, $\Qvee\subset
\LieT$ the coroot lattice, $Q\subset\LieT^{*}$ the root lattice, $\weight\supseteq
Q$ the weight lattice, and $\Weyl=N_G(T)/T$ the Weyl group. Because $G$ is
simply connected, the cocharacter lattice of $T$ is $\Qvee$, and we
identify $T=F/\Qvee$ using the normalized exponential map
$x\mapsto\exp_G(2\pi x)$ on the compact real Cartan. Thus the quotient
notation records the period lattice, including the $2\pi$ convention.
We write $\basicform{\slot}{\slot}$ for the basic invariant form, normalized so
that \emph{long roots} have squared length $2$ (equivalently, \emph{short
coroots} have squared length $2$). With this normalization $\Qvee$ is an
\emph{even} integral lattice: for a coroot $\alpha^\vee$ one has
$\basicform{\alpha^\vee}{\alpha^\vee}=4/\basicform{\alpha}{\alpha}\in\{2,2r\}$,
where $r\in\{1,2,3\}$ is the lacing number ($r=1$ precisely in the simply laced
case) and $2$ corresponds to the coroots of long roots (\ie the short coroots),
$2r$ to the coroots of short roots (\ie the long coroots); the coroot pairings
$\basicform{\alpha^\vee}{\beta^\vee}$ are integral. In particular every squared
length of $\Qvee$ is an even integer.

The \emph{level} in this note is a positive integer
$k\in\ZZ_{>0}$. We distinguish the basic form from the level form by
\[
  (x,y)_k:=k\basicform{x}{y}.
\]
Equivalently, the level specifies the positive characteristic
class $k\,c_G\in H^4(BG,\ZZ)$, where $c_G$ is the basic generator whose
torus restriction is $x\mapsto\basicform{x}{x}/2$ on $\Qvee$.
This is Henriques's integral-form description of levels
\cite{Hen17}*{Theorem~6}. In particular, $G$ is the group and $k$ is
the integer scaling the basic form; neither is part of the definition of
the other. For a simply connected semisimple product $G=\prod_jG_j$,
one instead chooses a tuple $k=(k_j)_j$ of positive integers and the
class $\sum_j k_jc_{G_j}$.

In the simply laced case $\rootsys=\Rvee$ and hence $Q=\Qvee$ inside $\LieT^*$
under the identification $\LieT\cong\LieT^{*}$ induced by the form; the
distinction between $Q$ and $\Qvee$ is then invisible. In the non-simply-laced
case $Q\neq \Qvee$ and the \emph{root} lattice $Q$ is \emph{not} even under the
basic form (short roots have squared length $2/r<2$), so $Q$ cannot serve as
the lattice seed; the coroot lattice $\Qvee$ is the correct choice.

Let $\cI$ be the set of proper intervals of $\SOne$ and fix $I\in\cI$ once and
for all.

\subsection{Loop groups, loop algebras, and the level extension}
\label{ssec:LoopGroupExtension}
We use upright $\LL$ for loop groups and loop algebras, and calligraphic
$\Lpc$ for the centrally extended groups.
For the group $G$ just specified, the \emph{smooth loop group} is
$\LL G=C^\infty(\SOne,G)$ with pointwise multiplication. Its Lie algebra
is $\LL\LieG=C^\infty(\SOne,\LieG)$ with pointwise bracket. The local
loop subgroup is
\[
 \LL_I G=\{\gamma\in\LL G:\gamma|_{I'}=e\}.
\]
All loops are smooth across the endpoints of $I$; thus extension by the
identity outside $I$ is smooth. The corresponding local loop algebra
consists of smooth Lie-algebra loops vanishing outside $I$.

A \emph{central extension} of $\LL G$ by $\TT$ is an exact sequence
\[
  1\longrightarrow\TT\longrightarrow\Lpc G
    \xrightarrow{p}\LL G\longrightarrow1
\]
in which the image of $\TT$ is central. A section of $p$ need not be a
group homomorphism; its multiplication defect is a scalar group
two-cocycle. Representations of the extension with the specified central
character are therefore projective representations of the loop group.

For each interval $I$ define the \emph{local central extension} by
\[
  \Lpc_I G:=p^{-1}(\LL_I G)\subset\Lpc G.
\]
Thus $\Lpc_I G$ is the inverse image of the local loop subgroup, not
a choice of section or of loop implementers, and fits into
\[
  1\longrightarrow\TT\longrightarrow\Lpc_I G
    \xrightarrow{p}\LL_I G\longrightarrow1.
\]
At level $k$ we write $\Lpc_I G_k=p^{-1}(\LL_I G)$ inside
$\Lpc G_k$; the projection $p$ is always that of the extension in use.

The basic form defines the affine Lie-algebra two-cocycle
\begin{equation}
  \label{eq:AffineLevelCocycle}
  \kappa_k(\xi,\eta)
     =\frac{1}{2\pi}\int_0^{2\pi}
         (\xi(\theta),\eta'(\theta))_k\dd\theta.
\end{equation}
Writing $c$ for the central Lie-algebra generator, the central extension
$\widehat{\LL\LieG}_k=\LL\LieG\oplus\RR c$ has bracket
\[
 [\xi+ac,\eta+bc]=[\xi,\eta]+\kappa_k(\xi,\eta)c.
\]
Antisymmetry follows by integration by parts on the circle, and the
cocycle identity follows from invariance of the basic form. The group
extension integrating the basic cocycle is the basic extension
$\Lpc G_1$; the integer $k$ scales its extension class
\cite{HenColimits}*{Section~2.1}, \cite{PS1986}*{Chapter~4}.
For the \emph{simple} group $G$ fixed in this subsection, the level-$k$
extension $\Lpc G_k$ is the pushout of the basic extension along
$\TT\to\TT$, $z\mapsto z^k$. Explicitly it is
\[
 \Lpc G_k=(\Lpc G_1\times\TT)/\{(z,z^{-k}):z\in\TT\}.
\]
The induced map $F_k\colon\Lpc G_1\to\Lpc G_k$, $a\mapsto[a,1]$,
sends a central element to
$F_k(z)=[z,1]=[1,z^k]$. Thus, if the central circle of $\Lpc G_k$
acts in $\pi_k$ by its defining scalar character, then the pulled-back
representation satisfies
\[
  (\pi_k\circ F_k)(z)=z^k\,1\qquad(z\in\TT).
\]
This is the reason for the central-character convention on the fixed
basic extension. It describes integer multiples of the \emph{chosen
simple-group basic class}, not a scalar parametrization of levels for
arbitrary compact groups.

\begin{def}\label{Profile}
A \textbf{profile function} is a smooth cutoff function $\chi \colon S^1 \to [0, 1]$ that equals $0$ on an open arc to the left of interval $I$, equals $1$ on an open arc to the right of $I$, and transitions smoothly between $0$ and $1$ inside $I$ and its complementary arc.
\end{def}

\smallskip
\noindent\textbf{Semisimple groups and tori.}
For $G=\prod_jG_j$, use the product of the basic extensions of the simple
factors. It has central torus $\TT^r$, and a level tuple $(k_j)_j$
corresponds to the character
\[
 (z_1,\ldots,z_r)\longmapsto\prod_j z_j^{k_j}.
\]
Pushout by this character gives the circle extension with the prescribed
tuple of levels. One fixed circle extension and the character $z^k$
cannot encode arbitrary independent levels on different factors.
For a torus $T=F/\Lambda$, the integral level datum is a symmetric
form $B$ on the cocharacter lattice with $B(x,x)\in2\ZZ$, equivalently
$q_B(x)=B(x,x)/2\in H^4(BT,\ZZ)$. There is no distinguished universal
one-dimensional basic ray for all torus levels. In
Proposition~\ref{prop:TorusLevelLattice}, the form is the particular
restriction $B_k=k\basicform{\slot}{\slot}|_F$ inherited from the simple
group $G$; scaling that fixed form is the only sense in which the torus
extension there carries a scalar $k$.

\subsection{Vacuum representations and loop group nets}
\label{ssec:LoopGroupNets}
Fix the positive integer $k$. A \emph{positive-energy representation}
has a compatible projective rotation action with generator bounded below.
The \emph{level-$k$ vacuum representation} $\pi_k$ is the unitary
integrable vacuum representation of $\Lpc G_k$, with affine
highest weight $k\Lambda_0$ when viewed as a representation of the
  basic affine algebra. Its Hilbert space is $\Hil_{G_k}$ and its vacuum
  vector is $\Omega_k$. Conformal covariance and local-algebra properties
  require the corresponding representation-theoretic results
  \cites{PS1986, Wa, HenColimits}.

\begin{defi}[Loop group conformal net]
\label{def:LoopGroupConformalNet}
Given such a loop-group representation, we call the local assignment
\[
  \A_{G_k}(I)\ :=\ \pi_k(\Lpc_I G_k)''
     \ \subset\ \B(\Hil_{G_k}).
\]
the \emph{loop group conformal net} associated with $G$ and level $k$:
its local algebras come from the local loop-group representation.
\end{defi}
The domain of $\pi_k$ is the central extension. In expressions involving
an individual loop, $\pi_k(\gamma)$ means
a chosen unitary lift of a loop; changing the lift changes only its scalar.

\begin{rmk}[Conformal-net axioms are additional assertions]
  \label{rmk:LoopGroupNetAxioms}
  The terminology specifies the loop-group origin of the local algebras.
  In general it is not clear from that construction alone that every
  property required of a conformal net is satisfied. Those properties
  must be established for the group, level and representation in use;
  they are not part of the definition above. Isotony follows from the
  local subgroup inclusions. Locality uses the disjoint-support
  commutation property of the chosen central extension; covariance,
  the vacuum properties and the further analytic net properties require
  their own results.

  For general simple groups see the representation comparisons in
  \cite{HenColimits}*{Section~3}. The construction in~\cite{Wa} concerns
  $\SU(N)$. The simply laced level-one lattice realization is treated
  in~\cite{Bi2011}.
  Hyperfiniteness, type \threeone, Haag duality and strong additivity
  likewise require the applicable analytic statements; the single
  $\SU(N)$ reference does not supply them for every type; in every
  simply connected type they are supplied by
  Corollary~\ref{cor:LoopGroupNetAnalytic}. Later arguments
  invoking locality, factoriality or a vacuum-net realization are made
  in the setting where those properties have been established.
\end{rmk}

\begin{cor}[Analytic properties of loop group nets]
  \label{cor:LoopGroupNetAnalytic}
  Let $G=\prod_jG_j$ be compact, connected, simply connected and
  semisimple with simple factors $G_j$, and let $k=(k_j)_j$ be positive
  integers. Then $\A_{G_k}$ is a completely rational conformal net. It
  coincides with the net of the affine VOA
  $V_{G,k}=\bigotimes_jV_{\mathfrak g_j,k_j}$, and its category of
  finite-index DHR sectors is braided equivalent to the module category
  of $V_{G,k}$. Every interval algebra $\A_{G_k}(I)$ is the hyperfinite
  type \threeone factor.
\end{cor}
\begin{proof}
  The vacuum representation of the product extension is the tensor
  product of the component vacuum representations, so
  $\A_{G_k}=\bigotimes_j\A_{(G_j)_{k_j}}$. For simple $G_j$ the affine
  VOA $V_{\mathfrak g_j,k_j}$ is strongly local and its net is the loop
  group net $\A_{(G_j)_{k_j}}$ \cite{CaKaLoWi2015}*{Theorem~8.3,
  Example~8.7}. Tensor products of strongly local VOAs are strongly local,
  with net the tensor product of the component nets
  \cite{CaKaLoWi2015}*{Corollary~8.2}; hence $\A_{V_{G,k}}=\A_{G_k}$.
  Complete rationality and the braided equivalence hold for finite tensor
  products of affine VOAs \cite{GuiUnbounded}*{Theorem~I}. The structural
  consequences of complete rationality are those of~\cite{KaLoMg2001}.

  The vacuum representation is irreducible, so each $\A_{G_k}(I)$ is a
  type \threeone factor \cite{DLR2001}*{Proposition~2.4}. The
  $L_0$-eigenspaces of the vacuum module have dimension at most the
  coefficients of $\prod_{m\ge1}(1-q^m)^{-\dim G}$, so
  $\Tr e^{-\beta L_0}<\infty$ for all $\beta>0$ and the net is split
  \cite{DLR2001}*{Theorem~3.2}; split also follows from diffeomorphism
  covariance alone \cite{MoTaWe2016}*{Theorem~5.4}. Split nets have
  hyperfinite interval algebras, being generated by an increasing
  sequence of intermediate type~I factors
  \cite{DLR2001}*{Proposition~3.1}.
\end{proof}

\subsection{The maximal-torus restriction and its lattice}
\label{ssec:TorusLevelLattice}
The inclusion $T\subset G$ gives $\LL T\subset\LL G$ at every level.
It is the \emph{pullback extension} of $\LL T$, not a new group called
``$G_k$'', that records the level on the torus.

\begin{prop}[Torus restriction of the level-$k$ extension]
  \label{prop:TorusLevelLattice}
  Let $G$ be compact connected simply connected simple, let
  $T=F/\Qvee$ be a maximal torus, and let $k\in\ZZ_{>0}$.
  Pullback of the level-$k$ extension gives a central extension
  \[
    1\longrightarrow\TT\longrightarrow\Lpc T_k
      \longrightarrow\LL T\longrightarrow1,
    \qquad \Lpc T_k\subset\Lpc G_k.
  \]
  Its integral level form on the cocharacter lattice is
  $B_k(x,y)=(x,y)_k=k\basicform{x}{y}$ on $\Qvee$. Consequently
  \[
    \Lambda_k=(\Qvee,B_k)
  \]
  is a positive even lattice, in every simple type. Equivalently one
  can use $L_k=\sqrt{k}\,\Qvee\subset F$ with the unscaled basic form.
  The torus map
  \[
    F/\Qvee\longrightarrow F/L_k,\qquad
    x+\Qvee\longmapsto\sqrt{k}x+L_k
  \]
  identifies the standard level-$k$ torus extension with the lattice
  torus extension for $L_k$. Its vacuum representation defines the
  lattice net $\A_{L_k}$, and its companion obstruction pulls back to
  the quadratic form $q_k(x)=k\basicform{x}{x}/2$ on $\Qvee$.
\end{prop}
\begin{proof}
  The pullback group is $p^{-1}(\LL T)$. Its loop-algebra cocycle is
  \eqref{eq:AffineLevelCocycle} restricted to $\LL\LieT$, so the
  invariant form is $B_k$. At the integral level, naturality of the
  characteristic-form identification~\cite{Hen17}*{Theorem~6} restricts
  $kc_G$ to the even form $B_k$ on $\Qvee$. The standard torus
  extension is the lattice extension, with winding lattice $\Qvee$
  and this level form; its full cocycle, including constant and winding
  terms, is described in~\cite{PS1986}*{Section~4.8} and
  \cite{DoXu2006}*{Section~3}. The assertion concerns that full standard
  extension, not just the Heisenberg cocycle of winding-zero loops.

  Evenness of the basic coroot lattice was checked above, so scaling
  by the positive integer $k$ preserves integrality, evenness and
  positivity. The map $x\mapsto\sqrt{k}x$ is an isometry from
  $(\Qvee,B_k)$ to $L_k$, and identifies their torus extensions.
  The lattice vacuum construction gives $\A_{L_k}$
  \cites{DoXu2006, Bi2011}. The companion anomaly for that lattice is
  $y\mapsto\basicform{y}{y}/2$; its pullback is $q_k$.
\end{proof}

This proposition identifies the \emph{extension and lattice data} of
the torus restriction. It does not equate the full representation
$\pi_k|_{\Lpc T_k}$ with one irreducible lattice vacuum:
torus sectors and multiplicities may occur in the restriction. Comparing
the scalar associator on the nonabelian factor requires compatible local
defects, as made explicit in Proposition~\ref{prop:ReduceCartan}.
In simply laced types at level one the compatible Frenkel--Kac--Segal
construction identifies the whole net with $\A_{\Qvee}$
\cite{Bi2011}*{Section~3.5}; this is a stronger, special statement.

\subsection{\texorpdfstring{$G$-kernels}{G-kernels} and Sutherland obstruction}
\label{ssec:Gkernels}
Following Sutherland~\cites{Su1980, Su1980-II} (see
also~\cite{BiKa}*{Section~2}), a \emph{$G$-kernel} on
a von Neumann algebra $M$ is a monomorphism $\theta\colon G\to\Out(M)$ that
admits a Borel lift $\alpha \colon G\to\AUT(M)$. Given a Borel lift
$\alpha_g$, set
\[
  \alpha_g\,\alpha_h\ =\ \Ad u(g,h)\,\alpha_{gh},
  \qquad u(g,h)\in\U(M).
\]
Choose the defect unitaries Borel measurably. When $M$ is a factor, define
the scalar associator by
\[
 \alpha_g(u(h,l))u(g,hl)=\omega(g,h,l)u(g,h)u(gh,l).
\]
The class
$[\omega]\in H^{3}_{\Borel}(G,\TT)$ is independent of the lift. This class is
the \emph{Sutherland obstruction} of the kernel. The characteristic-class identification
\begin{equation}
  \label{eq:Transgression}
  \tau\colon H^{4}(BG,\ZZ)\ \xrightarrow{\ \sim\ }\ H^{3}_{\Borel}(G,\TT),
\end{equation}
identifies Sutherland's $3$-cocycle with a characteristic class in
$H^{4}(BG,\ZZ)$ (see~\cite{WaWo2015}*{Remark~4.16} and the cohomology
discussion in~\cite{BiKa}). It is distinct from geometric transgression
to ordinary $H^3(G,\ZZ)$.

For a simply connected connected simple compact $G$, the classical computation
gives
\begin{equation}
  \label{eq:H4simple}
  H^{4}(BG,\ZZ)\ \cong\ \ZZ,
\end{equation}
the integer being the \emph{level}~$k$; for semisimple $G=\prod_i G_i$ with
simple factors $G_i$, $H^{4}(BG,\ZZ)\cong\bigoplus_i\ZZ$.

\begin{thm}[Integral degree-four Cartan restriction, Henriques]
  \label{thm:CartanRestriction}
  For a compact connected Lie group $G$ with maximal torus $T$ and Weyl
  group $\Weyl$, restriction identifies
  \[
    H^4(BG,\ZZ)\ \cong\ H^4(BT,\ZZ)^{\Weyl}.
  \]
  If $\Lambda$ is the cocharacter lattice of $T$, then
  $H^4(BT,\ZZ)=\Sym^2(\Lambda^*)$ identifies with even integral symmetric
  bilinear forms $B$ on $\Lambda$, or quadratic forms $q(x)=B(x,x)/2$.
\end{thm}
\begin{proof}
  This is \cite{Hen17}*{Theorem~6}; its proof includes the integral
  identification and naturality under group homomorphisms. In particular,
  this degree-four statement does not require $G$ to be simply connected.
\end{proof}

\section{Outer homomorphisms from jump loops}
\label{sec:kernels}

Throughout the section, $G$ and $T\subset G$ are as in
Section~\ref{sec:Notation}, and $k\in\NN$ is fixed.
\subsection{The \texorpdfstring{$G$}{G}-homomorphism from jump loops}
\label{ssec:Gkernel}

Fix $c<a<b<d$ on $\SOne$ and put $I=(a,b)$. Choose compact transition
regions in $I$ and in the complementary arc from $d$ to $c$. For $g\in G$,
choose a smooth path $p_g$ from $e$ to $g$, constant near both endpoints.
Following $p_g$ in the first transition region, and following it backwards
in the second, gives a smooth circle loop $\gamma_g$ equal to $e$ on
$(c,a)$ and to $g$ on $(b,d)$. Thus the return to $e$ is explicitly outside
$I$; we do not claim that a circle loop has only one transition.

The paths may be chosen Borel measurably: cover $G$ by finitely many
exponential coordinate neighborhoods of points $g_j$, choose a fixed path
from $e$ to each $g_j$, and concatenate it with the short exponential path
from $g_j$ to $g$. A fixed smooth reparametrization constant near the
endpoints makes the concatenation smooth. A Borel partition subordinate
to the cover selects one such path for each $g$; take $\gamma_e=1$.
Define, on $M_k=\A_{G_k}(I)$,
\[
  \alpha_{g,k}=\Ad U_{g,k}|_{M_k},\qquad U_{g,k}=\pi_k(\gamma_g),
\]
where unitary representatives are chosen Borel measurably using local
sections of the central extension. Scalar choices do not affect $\alpha_{g,k}$.
Then set
\begin{equation}
  \label{eq:GkernelDef}
  \theta_{g,k}\ :=\ [\alpha_{g,k}]\ \in\ \Out(M_k).
\end{equation}
\begin{prop}[Local composition defects]
  \label{prop:JumpHomomorphism}
  The maps $\alpha_{g,k}$ are normal automorphisms of $M_k$ and form a
  Borel lift of a homomorphism $\theta_k\colon G\to\Out(M_k)$.
  This outer homomorphism is independent of the chosen paths and returns
  with the prescribed constant values near the endpoints of $I$.
\end{prop}
\begin{proof}
  Pointwise conjugation by $\gamma_g$ preserves $\LL_I G$, so conjugation
  by $U_{g,k}$ preserves the von Neumann algebra it generates. Phases in
  the projective representation do not alter this inclusion. Applying the
  same argument to the inverse loop proves equality and hence invertibility;
  no identity $\gamma_g^{-1}=\gamma_{g^{-1}}$ is required.

  Put $\eta(g,h)=\gamma_g\gamma_h\gamma_{gh}^{-1}$. This loop is equal
  to $e$ on neighborhoods of both endpoints of $I$, but need not be $e$
  throughout $I'$. Define $\eta_I$ to agree with $\eta$ on $I$ and with
  $e$ on $I'$, and define $\eta_{I'}$ conversely. Both are smooth, and
  \[
    \eta=\eta_I\eta_{I'},\qquad
    \eta_I\in\LL_I G,\quad\eta_{I'}\in\LL_{I'}G.
  \]
  Consequently, up to a scalar,
  $\pi_k(\eta)=\pi_k(\eta_I)\pi_k(\eta_{I'})$. Locality makes
  $\pi_k(\eta_{I'})$ commute with $M_k$. Choosing Borel representatives
  \[
    u_k(g,h)=\pi_k(\eta_I(g,h))\in\U(M_k)
  \]
  gives, on $M_k$,
  \[
    \alpha_{g,k}\alpha_{h,k}=\Ad u_k(g,h)\,\alpha_{gh,k}.
  \]
  Passing to outer classes proves the homomorphism property.
  The Borel profile and section choices give Borel automorphism and defect
  maps in the usual strong operator topology.

  If $\gamma_g$ is replaced by another allowed profile, their quotient
  is again $e$ near the endpoints of $I$ and admits the same decomposition.
  Its $I$-part changes the automorphism by an inner automorphism of $M_k$;
  its $I'$-part acts trivially there. This proves independence in $\Out(M_k)$.
\end{proof}

\begin{rmk}[Faithfulness]
  \label{rmk:Faithfulness}
  Proposition~\ref{prop:JumpHomomorphism} proves a homomorphism, not a
  monomorphism. To call $\theta_k$ a $G$-kernel under our definition one
  must also prove that $\alpha_{g,k}$ inner implies $g=e$. This does not
  follow from scalarity of the associator or from flank triviality; it is
  proved in Theorem~\ref{thm:AnomalyAllTypes} from the nonvanishing of the
  class, Lemma~\ref{lem:NormalSubgroups} and
  Proposition~\ref{prop:CentreOuter}.
\end{rmk}

\subsection{The cut circle: a fundamental domain for the jump}
\label{ssec:FundamentalDomain}

The two-transition picture above hides the useful choice. Fix the
\emph{cut point} $t$ in the complementary arc $(d,c)$ and write
$\SOne_t=\SOne\setminus\{e^{\ima t}\}$ for the cut circle, identified with the
interval $(t,t+2\pi)$. A \emph{jump profile at $g$} is then a smooth map
$\gamma\colon\SOne_t\to G$ which equals $e$ near $t$ and near $a$, equals $g$
near $b$, and is otherwise a path from $e$ to $g$ supported in $I$: the return
to $e$ is absorbed at the cut rather than carried out inside $\SOne$.

This is exactly the device used for the torus kernel in~\cite{BiKa}: there
$m\in C^\infty(\SOne,F)$ has $\supp m\subset I$, a cutoff $f$ is supported in
$\tilde I\supset I$ with $f|_I=1$ and $f(t)=0$, and the primitive is taken
\emph{from the cut},
\[
  L(x)\ =\ f(x)\int_{t}^{x}m(y)\,\dd y ,
\]
the integral being taken in the contractible image of $(t,x)$ in $\SOne_t$.
Because $f$ vanishes near $t$, $L$ is a smooth function on the whole circle
with no jump, and $\Ad W([L])$ acts on the current net by the shift
\begin{equation}
  \label{eq:GaugeAutFormula}
  \alpha_m\big(W([g])\big)\ =\ \e^{\ima\int_{\SOne}\langle g(\theta),m(\theta) \rangle_F \frac{\dd\theta}{2\pi}}\,W([g]).
\end{equation}
Three facts, all proved in~\cite{BiKa}.

\begin{rmk}
  \label{lem:CutIndependence}
  The automorphism~\eqref{eq:GaugeAutFormula} depends only on the current shift
  $m$, not on the cut point $t$, the cutoff $f$, or the additive constant in
  the primitive. Changing these changes the chosen multiplier by a coboundary
  and the automorphism by an inner one.
  Formula~\eqref{eq:GaugeAutFormula} is expressed in $m$ alone and details can be found in \cite{BiKa}.
\end{rmk}

\begin{lem}[Zero shift is inner; \cite{BiKa}, `lem:InnerAuto'']
  \label{lem:ZeroShiftInner}
  If $\supp m\subset I$ and $\int_I m=0$, then $\alpha_m$ is inner on
  $\A(I)$.
\end{lem}
\begin{proof}
  Then $x\mapsto\int_a^x m$ is already supported in $I$, so its Weyl unitary
  lies in $\A(I)$ and implements $\alpha_m$.
\end{proof}

\begin{cor}[The outer class is the charge]
  \label{cor:ClassIsCharge}
  Two shifts supported in $I$ give the same outer automorphism of $\A(I)$
  \emph{iff} they have the same total charge $\int_I m$.
\end{cor}
\begin{proof}
  If the charges agree, their difference has zero charge and is inner by
  Lemma~\ref{lem:ZeroShiftInner}. Conversely a common inner conjugate
  preserves the charge; this is `equivlsamecharge' in~\cite{BiKa}. That
  argument uses \emph{strong additivity} of the current net
  $\A(I)=\A(I_1)\vee\A(I_2)$ after deleting an interior point to extend an
  intertwiner from $I$ to a larger interval before comparing charges. Strong
  additivity and its stability under compact-group fixed points are proved
  in~\cite{Xu2005} (see also~\cites{Lo1979, DoLo1984}).
\end{proof}

Thus in the torus setting the whole outer kernel is classified by
$q=\int_I m\in F$, \ie by $g=q\bmod L\in T=F/L$, exactly as in~\cite{BiKa},
and the choice of fundamental domain only moves the representative.

\begin{rmk}
  \label{rmk:CutScope}
  On the full circle a profile with $m\not\equiv0$ has winding and is not a
  single-valued smooth loop; the cut turns it into a function, and
  Lemma~\ref{lem:CutIndependence} shows the resulting outer class does not
  see the cut. Consequently $\alpha_m$ is the identity on $\A(J)$ for every
  interval $J\subset\SOne_t$ disjoint from $I$, but need not be trivial on
  intervals crossing $t$: it is a \emph{cut soliton}, not a DHR morphism of
  the uncut circle net.
\end{rmk}

\subsection{Flanks and the lattice discriminant}
\begin{lem}[Flank triviality]
  \label{lem:FlankTriviality}
  The action of $U_{g,k}$ on the left flank $\A_{G_k}(c,a)$ is trivial.
  Its action on the right flank $\A_{G_k}(b,d)$ is the constant gauge
  action of $g$, and is trivial if and only if $g\in Z(G)$.
\end{lem}
\begin{proof}
  Let $J\subset\SOne$ be an arc on which $\gamma_g$ is constant $=h$. For
  $x\in\LL_J G$ the product $\gamma_g x\gamma_g^{-1}$ is the pointwise
  conjugate $x^h$, so, for chosen projective representatives,
  \begin{equation}
    \label{eq:ConstantConjugation}
    \Ad U_{g,k}\big(\pi_k(x)\big)\ \sim\ \pi_k(x^h),\qquad x\in\LL_J G .
  \end{equation}
  Here $\sim$ records a possible scalar, not an ambiguity of the induced
  inner automorphism. More precisely, $\gamma_g h^{-1}$ is the identity
  on $J$, so locality identifies its action there with the constant-loop
  action of $h$. On the left flank $h=e$, giving the identity; on the right
  flank $h=g$, giving the constant gauge action.

  For necessity use the derived affine representation $\dd\pi_k$, rather
  than assuming a faithful representation of the unextended loop group.
  Triviality of the constant gauge action implies
  $\dd\pi_k((\Ad_gX-X)f)=0$ for every $X\in\LieG$ and real smooth
  $f$ compactly supported in the right flank. Set $Y=\Ad_gX-X$ and
  choose nonconstant $f$ and $h=f'$. The affine bracket of $Yf$ and
  $Yh$ has no loop-algebra part and has central part proportional to
  \[
    (Y,Y)_k\int f h'
       =-(Y,Y)_k\int (f')^2.
  \]
  The derived commutator on the common smooth-vector domain forces
  $Y=0$ since $k>0$. Thus $\Ad_g$ is the identity on $\LieG$, equivalent
  to $g\in Z(G)$ for connected $G$. Conversely a central constant loop
  acts trivially on the vacuum-generated local current algebra.
\end{proof}

\begin{rmk}
  Flank triviality is not a statement of DHR localization on the full
  circle: the return region also lies in the complement. A boundedly
  localized, transportable morphism requires a separate soliton/path
  construction. Nor does localization imply innerness on $M_k$.
\end{rmk}

\begin{prop}[Scaled discriminant and central subgroup in ADE]
  \label{prop:CentralDiscriminant}
  Suppose here that $G$ is simply laced. Identify $Q=\Qvee$ using the
  basic form and $Z(G)=P/Q$. For $L_k=\sqrt{k}\,Q$, one has
  \[
    L_k^*=\frac1{\sqrt{k}}P,\qquad
    L_k^*/L_k\cong P/kQ,\qquad
    |L_k^*/L_k|=k^{\operatorname{rank}G}|P/Q|.
  \]
  The map $p+Q\mapsto kp+kQ$ embeds $Z(G)$ in this discriminant group.
\end{prop}
\begin{proof}
  Integrality of $\basicform{x}{\sqrt{k}v}$ for all $v\in Q$ is
  equivalent to $\sqrt{k}x\in Q^*=P$. Multiplication by $\sqrt{k}$ gives
  the quotient identification, and the index formula follows from
  $[Q:kQ]=k^{\operatorname{rank}G}$. The asserted subgroup map is
  well defined because $kQ$ absorbs changes by $kq$, and injective because
  $kp\in kQ$ implies $p\in Q$.
\end{proof}
This is an algebraic subgroup identification; it does not identify the entire
discriminant with the sector set of a higher-level nonabelian net.

\subsection{Local sectors versus global gauge actions}
\label{ssec:GaugeVsLocal}

We record the two distinct notions of ``action'' that occur, since conflating
them obscures what has to be compared.

\smallskip
\noindent\textbf{Definitions.} An \emph{internal action} of a topological
group $H$ is given on the vacuum Hilbert space of $(\A,U,\Omega)$.
Following~\cite{BiDeGi2023}*{Section~2}, $\Aut(\A)$ consists of
$\alpha=\Ad V$, where $V\in\U(\Hil)$ satisfies
\[
 \begin{aligned}
   V\Omega&=\Omega,\qquad V\A(I)V^*=\A(I) &&\text{for every }I,\\
   VU(\gamma)&=U(\gamma)V &&\text{for every }\gamma.
 \end{aligned}
\]
Here $\gamma$ ranges over $\Diff_+(\SOne)$; projective phases of
$U(\gamma)$ do not affect commutation. This is a group of global
vacuum-preserving net symmetries, distinct from
$\Aut(\A(I))$ for a single factor. Its topology is the strong operator
topology on the vacuum-preserving implementing unitaries. The unitary
implementer is unique: it is determined by $Vx\Omega=\alpha_I(x)\Omega$.
Equivalently one can use compatible vacuum-state-preserving local
automorphisms, as in the companion paper's subsection on topology of sectors.

A compact group $H$ \emph{acts properly} if there is a faithful strongly
continuous unitary representation $V\colon H\to\U(\Hil)$ with
$V(h)\Omega=\Omega$ and $V(h)\A(I)V(h)^*=\A(I)$ for every $h,I$
\cite{Xu2005}*{Definition~2.1, Section~2.3}. Commutation with M\"obius
covariance follows from these conditions; the conformal-net formulation
above also records diffeomorphism covariance~\cite{BiDeGi2023}*{Section~2}.
The fixed-point subnet on $\Hil$ is
$\B(I)=\A(I)\cap V(H)'$. Its vacuum realization is obtained by
restricting to $\Hil^H=\overline{(\bigvee_I\B(I))\Omega}$; this is
Xu's orbifold net $\A^H$.

DHR localization instead concerns transportable representations or
endomorphisms, not these global vacuum-preserving symmetries. A DHR
endomorphism localized in $I$ is the identity on $\A(J)$ for $J\subset I'$.
A twisted representation on a cut circle is the identity on the left of
$I$ and a specified global symmetry on the right. It extends to an
untwisted DHR representation only when its boundary twist is trivial,
with transportability also required. The DHR framework originates
in~\cite{DoHaRo1969II}; the chiral statistics are braided
\cite{FrReSc1989}. Doplicher--Roberts reconstruction~\cite{DoRo1989}
applies to the symmetric finite-statistics setting, and is not applied
here to reconstruct a compact gauge group from an arbitrary braided
chiral category. For the infinite-index extensions used here the
appropriate extension framework is~\cite{DeGi2018}; generalized
fixed-point symmetries are treated in~\cite{BiDeGi2023}.

\smallskip
\noindent\textbf{The lattice net has two layers.} Let $L\subset F$ be even,
$\A_F\subset\A_L=\A_F\rtimes_b L$, and let $\psi_\ell$ ($\ell\in L$) be the
charged unitaries. For $q\in F$ let $\beta_q$ be the
$\alpha$-induced lift of the current charge $q$ on $\A_L(I)$.
Then:
\begin{enumerate}\itemsep2pt
\item $\beta_q$ is \emph{inner} on $\A_L(I)$ iff $q\in L$; the lattice
  charge $q\in L$ is implemented by $\psi_q$;
\item the induced representation has left boundary action $\id$ and
  right boundary action $\delta_{q+L^*}$, where $\delta$ is the global
  action below. The twist is trivial precisely when $q\in L^*$;
  these are the lattice DHR sectors $L^*/L$;
\item the extension has an \emph{honest} internal dual gauge action of
  \[
    T_{\mathrm{gauge}}\ =\ F/L^{*},
  \]
  whose fixed-point subnet is $\A_F$, with its vacuum realization as
 described. It is the dual action of the lattice crossed product;
Explicitly, in the companion's sign convention,
\[
 \delta_{f+L^*}(\iota(a)\psi_\ell)
    =e^{-2\pi\ima\langle f,\ell\rangle}\iota(a)\psi_\ell.
\]
This fixes the current subnet pointwise and rotates the charged fields.
The kernel of this action of $F$ is exactly $L^*$, so its faithful
acting group is $F/L^*\cong\widehat L=\Hom(L,\TT)$. Haar averaging
retains exactly the zero-charge Fourier component, giving
$\A_L^{F/L^*}=\A_F$. Vacuum implementation is diagonal by these characters
on $\bigoplus_{\ell\in L}(\Hil_F)_\ell$ and fixes the charge-zero vacuum.
These formulas are clearly described in twisted sectors section \cite{BiKa}.
\end{enumerate}

\smallskip
\begin{rmk} The companion's faithful $F/L$-kernel
is a homomorphism into $\Out(\A_L(I))$, not a vacuum-preserving action
on the net. Its twisted sectors have the boundary-degree sequence
\begin{equation}
  \label{eq:TwoTori}
  1\longrightarrow \underbrace{L^{*}/L}_{\text{local / DHR}}
  \longrightarrow \underbrace{F/L}_{T}
  \longrightarrow \underbrace{F/L^{*}}_{T_{\mathrm{gauge}}}
  \longrightarrow 1 .
\end{equation}
The kernel $L^{*}/L$ consists of untwisted DHR sector classes; projection
records the right boundary global symmetry. This is not a restriction
of an honest $F/L$ net action. In the simply laced case, for $L=\Qvee=Q$
under the basic form, $L^{*}=P$ and
\eqref{eq:TwoTori} is exactly
\[
  1\longrightarrow Z(G)=P/\Qvee\longrightarrow T
  \longrightarrow T/Z(G)=T_{\mathrm{ad}}\longrightarrow 1 ,
\]
the maximal torus of $G$ and of $G_{\mathrm{ad}}$. The Sutherland
obstruction is the scalar associator class of the outer kernel, obstructing
the implementing twisted crossed product. It is \emph{not} merely the
obstruction to splitting the ordinary group extension
\eqref{eq:TwoTori}; those are different extension problems.
\end{rmk}

\begin{rmk}[Status of the identification]
  \label{rmk:TwoToriStatus}
  The companion's induced-kernel calculation and twisted-sector
  construction supply (1)--(2); the explicit dual action supplies (3).
  The continuous gauge group is $\widehat L$, whereas the finite group
  $L^*/L$ labels DHR sectors. Infinite index does not make this explicit
  group action into a hypergroup: more general inclusions can have
  hypergroup symmetries, but this example has a specified compact group
  action. Outside ADE the coroot-lattice discriminant must not be
  identified with the center without checking the weight/coweight data.
\end{rmk}

\subsection{Twisted representations and their boundary degree}
\label{ssec:TwistedRepresentations}
Let $G$ act properly on $\A$ by global symmetries $\delta_g$ as above.
Cut the circle at $p\notin\overline I$ and order its intervals as on
the real line. In M\"uger's endomorphism formulation~\cite{Mug05}*{Section~2.1},
a locally normal endomorphism $\rho$ of the algebra generated by the
bounded interval algebras is \emph{$g$-localized in $I$} if
\[
  \rho|_{\A(J)}=\id\quad(J<I),\qquad
  \rho|_{\A(J)}=\delta_g\quad(J>I).
\]
It is \emph{transportable} if it is unitarily equivalent to a
$g$-localized endomorphism in every other allowed interval. These
transportable morphisms are the $g$-twisted representations; for $g=e$
one recovers the untwisted localized representation category. Neither
the vacuum-preserving property nor invertibility is part of this
definition of a twisted representation.

The boundary degree is $\partial\rho=g$. M\"uger's action hypotheses
include local faithfulness, so the degree is unique. Composition gives
$\partial(\rho\sigma)=(\partial\rho)(\partial\sigma)$; global conjugation
by $\delta_h$ sends degree $g$ to $hgh^{-1}$. The associated category
is braided crossed: for a homogeneous object $\rho$ of degree $g$,
the exchange morphism has the form
$\rho\sigma\to\delta_g\sigma\delta_g^{-1}\rho$, not in general an
ordinary braiding $\rho\sigma\to\sigma\rho$. The untwisted component
has the usual DHR braiding. M\"uger's existence results for all degrees
in the completely rational finite-group setting do not automatically
give existence or invertibility in every degree for a continuous group.

In the companion lattice example the global symmetry group is
$G=F/L^*$. The induced sector represented by $q+L\in F/L$ has degree
$q+L^*$; the degree-zero classes are $L^*/L$. This is the companion's
subsection on $F/L^*$-twisted kernels and its proposition `prop:Mueger'.
Restriction of invertible twisted sectors to a fixed local factor
provides the outer-kernel viewpoint, whereas the boundary degree records
the global symmetry on the right. These two maps are not the same map.

\section{The \texorpdfstring{$\SU(2)$}{SU(2)} reduction and the level-one lattice}
\label{sec:SU2}

We elaborate for $G=\SU(2)$, $k=1$, where
the relevant lattice is $A_1$ and the Cartan is $T_L\cong\U(1)$. This is the
rank-one instance and the abelian seed of~\cite{BiKa}.

\subsection{The \texorpdfstring{$A_1$}{A1} lattice net}
\label{ssec:A1setup}
Let $L=A_1=\sqrt2\,\ZZ\subset\RR$ with $\langle x,x\rangle=x^2$, dual lattice
$L^{*}=\tfrac1{\sqrt2}\ZZ$, and discriminant form $D=L^{*}/L\cong\ZZ/2$. The torus
$T_L=\RR/L\cong\U(1)$ is a maximal torus of $\SU(2)$, with Weyl group
$\Weyl\cong\ZZ/2$ acting by $x\mapsto -x$. The lattice net $\A_L$ of
\cite{DoXu2006} (see also~\cites{Bi2011, BiKa}) is, by the algebraic
Frenkel--Kac--Segal construction~\cite{Bi2011}*{Section~3.5}, the level-$1$ loop group net
$\A_{\SU(2)_1}$:
\[
  \A_L\ \cong\ \A_{\SU(2)_1}.
\]
Its DHR sectors are $\{[\id],[\rho]\}\cong D$, with $[\rho]^2=[\id]$ and
$h_\rho=\tfrac14$~\cite{DoXu2006}. The local lattice algebra is the hyperfinite
type \threeone factor.

\subsection{The companion torus kernel}
\label{ssec:TorusKernel}
The companion paper~\cite{BiKa} constructs a faithful $T_L=F/L$-kernel
using $\alpha$-induction for the current-net extension
$\A_F\subset\A_L$, and computes its scalar associator. Write its
obstruction as $[\omega_L]\in H^3_{\Borel}(T_L,\TT)$, rather than
identifying a two-variable Weyl multiplier with a three-variable cocycle.
Under~\eqref{eq:Transgression}, the class is the integral quadratic form
\begin{equation}
  \label{eq:A1form}
  q_L(x)=\tfrac12\langle x,x\rangle_{A_1},\qquad
  q_L\in H^4(BT_L,\ZZ)=\Sym^2(L^*).
\end{equation}

\subsection{The comparison target on generators}
\label{ssec:IndUpCartan}

Let $F=L\otimes_\ZZ\RR$, let $\iota\colon\A_F(I)\to\A_L(I)$ be the
current-subnet inclusion, and write $\psi_\nu$ ($\nu\in L$) for charged
unitaries satisfying $\psi_\nu\iota(a)\psi_\nu^*=\iota(\rho_\nu(a))$.
In the convention of the companion paper, the induced lift $\beta_q$
of a current charge $q\in F$ has the generator formula
\[
 \beta_q(\iota(a)\psi_\nu)
 =e^{-\pi\ima\langle q,\nu\rangle_F}\iota(\rho_q(a))\psi_\nu.
\]
This is the braiding-based extension formula of~\cite{BiKa}; its outer
class depends on $q$ modulo $L$. 

\begin{rmk}[The half-factor convention]
  \label{rmk:HalfFactor}
  The $\tfrac12$ in the phase is a variable-limit integral, not a product of
  two total integrals. Put $\dd\mu=\dd\theta/(2\pi)$, choose $\ell_0$
  supported in $I$ with $\int_I\ell_0\dd\mu=1$, and set
  $F_0(\theta)=\int_a^\theta\ell_0(t)\dd\mu(t)$. Then
  \[
    \int_I\ell_0F_0\dd\mu=\int_I F_0\dd F_0=\tfrac12 .
  \]
  If the primitive is taken with respect to $\dd\theta$ instead of $\dd\mu$,
  the corresponding $2\pi$ must be kept; the two conventions must not be
  mixed within one computation.
\end{rmk}

\subsection{Comparison with the torus kernel}
\label{ssec:TorusComparison}

The braiding $\varepsilon$ entering the generator formula is the
Fredenhagen--Rehren--Schroer braiding of the localized sectors of the
current net~\cite{FrReSc1989}; it is exactly this braiding that the
companion paper uses to extend automorphisms through the crossed product,
and the extension formula is legitimate because the sectors carry braid
group statistics~\cite{FrReSc1989}. The correct notion of \emph{the same}
kernel on \emph{the same} local algebra is equality of the braided action of
the DHR category on that algebra, not merely isomorphism of abstract sector
groups~\cite{GiRe2018}.

The torus representation has an explicit decomposition. In the normalized
circle coordinate $u\in[0,1]$, an honest torus loop has a lift
$f(u)=\Delta_f u+f_0+f_1(u)$, where $\Delta_f\in L$ is its winding,
$f_0$ is the constant mode, and $f_1$ is periodic with zero mean.
Up to scalar section choices, its implementer has the form
\begin{equation}
  \label{eq:LoopUnitary}
  V(f)\ \sim\ Z(f_0)\,W([f_1])\,\widetilde\Gamma_{\Delta_f}.
\end{equation}
Here $Z(f_0)$ is the constant-loop implementer, diagonal on charge
spaces, and $\widetilde\Gamma_{\Delta_f}$ shifts charge by the winding,
with the lattice cocycle included. This is the construction in
\cite{Bi2011}*{Section~3.4}, equivalently the torus representation
description in~\cite{DoXu2006}*{Section~3}. The winding operator cannot
be discarded for a nonzero-winding loop. Moreover a jump parameter
$q\in F$ need not lie in $L$; a cut profile of arbitrary charge is not
therefore automatically an honest circle loop of winding $q$.

The comparison must recover the induced generator formula above on both
the current algebra and the charged fields, retaining the constant and
winding operators. Agreement on the current algebra by itself does not
determine the action on $\psi_\nu$. The decomposition makes the
representation-level comparison concrete; it is not a proof that the
arbitrary-charge cut profile already has the desired induced action.

\begin{prop}[Torus jumps and lattice $\alpha$-induction]
  \label{prop:TorusLoopInduction}
  Let $L\subset F$ be a positive even lattice, and realize $\A_L$ by
  the vacuum representation $V$ of the central extension of
  $L(F/L)$~\cites{DoXu2006, Bi2011}. Fix $I$ and choose a Borel section
  $s\colon F/L\to F$. For $q=s(t)$ choose the torus jump profile below.
  On $\A_L(I)$ its conjugation agrees with the companion's
  $\alpha$-induced lift $\beta_q$, with the same braiding orientation.
  Hence the torus jump construction and the companion construction give
  the same outer kernel and the same Sutherland obstruction on this factor.
\end{prop}
\begin{proof}
  Use the normalized circle coordinate $u=\theta/(2\pi)$ and measure
  $\dd u$. Choose a real smooth function $\ell$ compactly supported in
  $I=(a,b)$ with $\int_I\ell\dd u=1$, and put
  $H(u)=\int_a^u\ell(v)\dd v$ on $I$. Choose a smooth periodic
  $F$-valued function $f_q$ equal to $qH$ on $I$, zero on the left
  flank, and $q$ on the right flank, with its return to zero outside
  $I$. Then $\gamma_t=e^{2\pi\ima f_q}$ is an honest torus loop of
  winding zero even when $q\notin L$. The Borel section chooses its
  endpoint representative, not a nonintegral winding.

  We use the lifted torus-loop commutator from
  \cite{DoXu2006}*{Section~3}. For a winding-zero real lift $f$ and
  a local torus loop with lift $h$ of winding $\nu\in L$, it reads
  \[
    V(e^{2\pi\ima f})V(e^{2\pi\ima h})V(e^{2\pi\ima f})^*
    =e^{-2\pi\ima\int\langle h',f\rangle\dd u}
      V(e^{2\pi\ima h}).
  \]
  This includes the constant-mode/winding pairing: it is not just the
  oscillator commutator. For a current Weyl generator $W([g])$,
  take $h=g/(2\pi)$, so integration by parts gives
  \[
    \Ad V(\gamma_t)(\iota(W([g])))
      =e^{\ima\int_I\langle g,q\ell\rangle\dd u}
         \iota(W([g]))
      =\iota(\rho_q(W([g]))).
  \]
  Thus the current charge is $q$: the unnormalized logarithm is
  $2\pi f_q$, whose derivative with respect to $\theta$ is
  $q\ell$ on $I$, consistently with the companion's charge convention.

  Choose the charged generators $\psi_\nu$ as local loop implementers
  with winding $\nu$ and derivative $h_\nu'=\nu\ell$ on $I$, including
  their lattice cocycle factors. Such generators give the crossed-product
  realization~\cite{Bi2011}*{Section~3.4}. Since $f_q=qH$ on $I$,
  \[
    \int_I\langle h_\nu',f_q\rangle\dd u
       =\langle\nu,q\rangle\int_I\ell H\dd u
       =\tfrac12\langle\nu,q\rangle.
  \]
  The same commutator therefore gives
  $\Ad V(\gamma_t)(\psi_\nu)
       =e^{-\pi\ima\langle q,\nu\rangle}\psi_\nu$.
  Together with the current formula, this is exactly the companion's
  induced formula on $\iota(a)\psi_\nu$. These elements generate
  $\A_L(I)$, so normality gives equality of the automorphisms there.
\end{proof}

\begin{prop}[Reduction to the torus kernel]
  \label{prop:ReduceCartan}
  Let $T\subset G$ be a maximal torus and put $L=\Qvee$ with the basic
  form. Suppose that there is a unital normal embedding
  $\jmath\colon N=\A_L(I)\to M_1=\A_{G_1}(I)$ and Borel automorphism
  lifts $\beta_t$ of the companion torus kernel and $\alpha_{t,1}$ of
  $\theta_1|_T$, with Borel unitaries $v(t,s)\in N$, such that
  \[
    \alpha_{t,1}\jmath=\jmath\beta_t,\qquad
    \beta_t\beta_s=\Ad v(t,s)\,\beta_{ts},\qquad
    \alpha_{t,1}\alpha_{s,1}=\Ad\jmath(v(t,s))\,\alpha_{ts,1}.
  \]
  Then the restriction of the ambient obstruction to $T$ equals the
  companion obstruction $[\omega_L]$. In particular, under
  \eqref{eq:Transgression}, it restricts to
  $q_L(x)=\basicform{x}{x}/2$ on $\Qvee$.
\end{prop}
\begin{proof}
  The companion lifts have scalar associator
  \[
    \beta_t(v(s,r))v(t,sr)
      =\omega_L(t,s,r)v(t,s)v(ts,r).
  \]
  Applying $\jmath$ and using the intertwining identity gives exactly
  the associator equation for the ambient lifts with defects
  $\jmath(v(t,s))$. The embedding is unital, so the scalar remains
  $\omega_L(t,s,r)$. Hence the scalar cocycles agree for these choices,
  and arbitrary Borel choices give the same cohomology class. The
  quadratic-form identification is the companion anomaly computation.
\end{proof}

\begin{rmk}[The outstanding comparison]
  \label{rmk:TorusComparison}
  The hypothesis of Proposition~\ref{prop:ReduceCartan} is a level-one
  ambient-factor comparison. Proposition~\ref{prop:TorusLoopInduction}
  supplies the torus-loop/$\alpha$-induction comparison on the lattice
  factor. Agreement on the current subnet alone does not prove transfer:
  the \emph{same local defect unitaries} must implement composition on
  the ambient factor. The cut profiles, charged fields, zero modes, and
  braiding choices specify the comparison. An equivalent
  comparison on the whole factor also suffices, taking $N=M_1$.

  No ADE hypothesis is needed to define the candidate lattice:
  $L=\Qvee$ is even integral for the basic form in every simple type.
  What is special to ADE at level one is equality of the whole nonabelian
  net with the lattice net~\cite{Bi2011}*{Section~3.5}. Outside ADE that
  equality is not needed for the scalar calculation above; a compatible
  embedding and common defects suffice. Establishing those data, or an
  equivalent direct associator comparison, is the remaining analytic
  step. Higher levels will then follow from
  Proposition~\ref{prop:DiagonalAdditivity}.

  The statement about representations needed here is compatibility with
  the torus loop-group realization, not equality of the full restricted
  $LG$ vacuum with one irreducible $LT$ vacuum. Torus-loop representations
  have sectors indexed by $L^*/L$, with
  $\Hil_{L,\lambda}=\bigoplus_{\mu\in\lambda+L}(\Hil_F)_\mu$
  \cite{DoXu2006}*{Section~3}. A restriction may contain multiple such
  sectors and multiplicities. Compatible lifts and defects, as in
  Proposition~\ref{prop:ReduceCartan}, are sufficient for transferring
  the scalar associator even when the Hilbert-space realizations differ.
\end{rmk}

\begin{lem}[Faithfulness detected by a maximal torus]
  \label{lem:TorusFaithfulness}
  If $G$ is compact connected and $\theta\colon G\to H$ is a group
  homomorphism whose restriction to a maximal torus is injective, then
  $\theta$ is injective.
\end{lem}
\begin{proof}
  Every $g\in G$ is conjugate to an element $t\in T$. If $\theta(g)=e$,
  conjugating gives $\theta(t)=e$, whence $t=e$ and $g=e$.
\end{proof}

\subsection{The level-one anomaly}
\begin{proof}[Proof of Theorem~\textup{\ref{thm:SU2main}}]
  By the compatible lattice/loop-group realization
  $\A_L\cong\A_{\SU(2)_1}$~\cite{Bi2011}*{Section~3.5}, the restriction
  of $\pi_1$ to the maximal-torus loop group is the lattice vacuum
  representation $V$ on the same Hilbert space. For $t\in T_L$ we may
  choose the torus-valued profiles of
  Proposition~\ref{prop:TorusLoopInduction}: profile independence in
  Proposition~\ref{prop:JumpHomomorphism} identifies their outer classes
  with those of the original selector. Proposition~\ref{prop:TorusLoopInduction}
  then identifies $\theta|_{T_L}$ with the companion kernel on
  $\A_L(I)$, including its action on the charged generators.
  The torus kernel of~\cite{BiKa} is faithful, so
  Lemma~\ref{lem:TorusFaithfulness} makes $\theta$ faithful. Naturality of
  the obstruction and the torus comparison just proved give
  \[
    \tau^{-1}([\omega_\theta])|_{BT_L}=q_L.
  \]
  Theorem~\ref{thm:CartanRestriction} determines the full class uniquely.
  The A1 form $q_L(\sqrt2 n)=n^2$ is the restriction of the positive
  basic generator. With the standard torus character class $t$, the
  defining SU(2) bundle has $c_2|_{BT}=-t^2$, so this convention calls
  $-c_2$ the positive generator. Thus the obstruction is the level-one
  class. This argument computes the obstruction from restriction; it
  does not rederive the already specified representation level.
\end{proof}
The same argument applies to every simply connected simply laced simple
group at level one using its root/coroot lattice and the compatible
Frenkel--Kac--Segal realization. Outside ADE,
Proposition~\ref{prop:TorusLoopInduction} remains a lattice-net statement;
the ambient-factor transfer in Proposition~\ref{prop:ReduceCartan} is
the extra input this route would need. Theorem~\ref{thm:AnomalyAllTypes}
obtains the class without it.

\begin{rmk}[$Q$-system and induction]
  \label{rmk:Qsystem}
  The lattice extension constructs the net; $\alpha$-induction constructs
  morphisms or sectors on that extension. The nontrivial A1 discriminant
  label has $h_\rho=1/4$, hence twist $e^{2\pi\ima h_\rho}=\ima$ in this
  orientation. A charged commutation sign $-1$ is not the same quantity
  as this twist. These sector calculations inform the comparison in
  Proposition~\ref{prop:ReduceCartan}; they do not replace it.
\end{rmk}

\section{General simply connected, simple or semisimple cases}
\label{sec:General}

We now isolate the level-one input for all simple types. Let
$c_{\mathrm{basic}}\in H^4(BG,\ZZ)$ be the positive basic generator. Its
restriction is the quadratic form on $L=\Qvee$ by
Theorem~\ref{thm:CartanRestriction}. Thus the level-one comparison in
Proposition~\ref{prop:ReduceCartan} determines the whole class, including
outside ADE. Once the candidate $\theta_1$ is faithful, the level-$k$
vacuum construction is handled by the tensor argument below; no new torus
calculation at each level is required.

We state the resulting general conditional identification, which specializes to
Theorem~\ref{thm:SU2main} when $G=\SU(2)$, $k=1$.

\begin{thm}[Anomaly identification from the torus restriction]
  \label{thm:main}
  Let $G$ be a compact, simply connected, connected, simple Lie group, and let
  $k\ge1$. Let $\theta_k\colon G\to\Out(M_k)$ be the homomorphism of
  Proposition~\ref{prop:JumpHomomorphism}, where $M_k=\A_{G_k}(I)$ is the
  hyperfinite type \threeone factor (Corollary~\ref{cor:LoopGroupNetAnalytic}).
  Suppose that $\theta_k|_T$ is injective and that its obstruction restricts,
  under~\eqref{eq:Transgression}, to the quadratic form
  $q_k(x)=\frac{k}{2}\basicform{x}{x}$ on $\Qvee$.
  Then $\theta_k$ is a $G$-kernel and its Sutherland obstruction
  \[
    [\omega_{\theta_k}]\in H^3_{\Borel}(G,\TT)\ \cong\ H^{4}(BG,\ZZ)\ \cong\ \ZZ
  \]
  is $k$ times the positive basic generator, realized on the hyperfinite
  type \threeone factor.
  More generally, for $G=\prod_j G_j$, write $k=(k_j)_j$, choose the
  component kernels satisfying these hypotheses, and use the tensor-product
  factor. Its obstruction is the tuple $(k_j)_j$.
\end{thm}
By Theorem~\ref{thm:AnomalyAllTypes} and naturality of the obstruction,
both hypotheses hold a posteriori in every type. Theorem~\ref{thm:main}
records the torus route, which is independent of the restriction argument
only in types $A$, $D$, $E$.
\begin{proof}[Proof of Theorem~\textup{\ref{thm:main}}]
  Lemma~\ref{lem:TorusFaithfulness} proves injectivity. The assumed torus
  obstruction and naturality give
  $\tau^{-1}([\omega_{\theta_k}])|_{BT}=q_k$.
  By Theorem~\ref{thm:CartanRestriction}, restriction is injective, and
  $q_k$ is the restriction of $k$ times the basic class in our convention.
  This proves the simple case.

  For the product case, on $M=\bar\otimes_j M_{k_j}$ choose
  $\alpha_g=\otimes_j\alpha^{(j)}_{g_j}$ and
  $u(g,h)=\otimes_j u_j(g_j,h_j)$. Their associator is
  \[
    \omega(g,h,l)=\prod_j\omega_j(g_j,h_j,l_j).
  \]
  Restriction to each component gives its stated class, yielding the tuple
  under $H^4(BG,\ZZ)=\bigoplus_jH^4(BG_j,\ZZ)$.
  The product outer homomorphism is faithful: if the tensor automorphism is
  inner, a nonzero normal slice of its implementing unitary over all but one
  factor intertwines that component automorphism with the identity. In a
  factor the polar decomposition of such a nonzero intertwiner gives a
  unitary implementer. Thus every component is inner, and component
  faithfulness forces every $g_j=e$. The tensor-product factor is
  $\A_{G_k}(I)$ for the product group, hence hyperfinite of type
  \threeone by Corollary~\ref{cor:LoopGroupNetAnalytic}.
\end{proof}

\begin{prop}[Tensor associators and diagonal additivity]
  \label{prop:DiagonalAdditivity}
  For automorphism lifts $\alpha_g$ on a factor with inner defects $u(g,h)$,
  their $r$-fold tensor powers have scalar associator $\omega(g,h,l)^r$.
  Hence their obstruction class is $r[\omega]$ in additive notation.

  For compact simply connected simple $G$, apply this to the level-one
  jump-loop construction, using the same profiles at every level. Realize
  the level-$k$ vacuum representation on the invariant vacuum
  subrepresentation of the diagonal representation $\pi_1^{\otimes k}$.
  Then the local-defect construction gives
  \[
    [\omega_{\theta_k}]=k[\omega_{\theta_1}].
  \]
  This equality of obstruction classes does not require a new level-$k$
  torus-comparison calculation.
\end{prop}
\begin{proof}
  Tensor the identity defining $\omega$ in Section~\ref{ssec:Gkernels}
  $r$ times. The scalar on the right is multiplied $r$ times. Equivalently,
  pull back the external-product class along the diagonal $G\to G^r$.

  For the second assertion, the tensor vacuum is an affine highest-weight
  vector of weight $k\Lambda_0$ for the diagonal action. Its generated
  unitary vacuum subrepresentation is the level-$k$ vacuum representation;
  this is the usual affine highest-weight realization
  \cite{PS1986}*{Chapter~9}. Let $\Hil_k\subset\Hil_{G_1}^{\otimes k}$
  denote this invariant subspace and choose compatible representatives of
  the level-$k$ central extension. Both
  \[
    U_{g,1}^{\otimes k}=\pi_1(\gamma_g)^{\otimes k},\qquad
    u_1(g,h)^{\otimes k}=\pi_1(\eta_I(g,h))^{\otimes k}
  \]
  preserve $\Hil_k$, since they represent actual loop-group elements.
  Their restrictions are the level-$k$ lift implementers and local
  defect unitaries, respectively, up to the scalar choices of section.
  The latter restrictions lie in $M_k$ because $\eta_I\in\LL_I G$.
  Restricting the tensor associator identity to $\Hil_k$ therefore gives
  the scalar $\omega_{\theta_1}(g,h,l)^k$. Section changes contribute
  coboundaries, proving the class equality. This argument restricts
  an explicit operator identity; it does not infer equality of obstruction
  classes merely from the existence of a subnet.
\end{proof}
\begin{rmk}
  The assertion used here is that the operators restrict to the invariant
  vacuum subspace, not that $\pi_k$ is the entire tensor representation.
  Additivity still does not normalize the base class or establish
  faithfulness. For a simple group, level-one torus comparison identifies
  $[\omega_{\theta_1}]$ with the single basic generator; the preceding
  calculation then identifies $[\omega_{\theta_k}]$ with $k$ times that
  generator. Injectivity of the higher-level outer homomorphisms is a
  separate statement, proved in Theorem~\ref{thm:AnomalyAllTypes}.
\end{rmk}

\begin{cor}[Simply laced cases]
  \label{cor:SimplyLaced}
  Fix $G$ simply connected, connected, simply laced simple (any of $A_n$,
  $D_n$ with $n\ge4$, $E_6$, $E_7$, $E_8$), and $k\ge1$. There is a
  faithful $G$-kernel on $\A_{G_1}(I)^{\bar\otimes k}$ whose obstruction
  is $\tau(kc_G)$. In particular this realizes every positive multiple
  of the basic class on the hyperfinite type \threeone factor.
\end{cor}

\begin{proof}[Proof of Theorem~\ref{thm:SimplyLacedRealization} and
  Corollary~\ref{cor:SimplyLaced}]
  Put $L=Q=\Qvee$. The compatible level-one lattice realization of
  \cite{Bi2011}*{Section~3.5} and
  Proposition~\ref{prop:TorusLoopInduction} identify the maximal-torus
  restriction of $\theta_1$ with the faithful companion kernel on
  $\A_L(I)=\A_{G_1}(I)$. Lemma~\ref{lem:TorusFaithfulness} gives
  faithfulness on $G$. The companion anomaly is the basic lattice
  quadratic form; Theorem~\ref{thm:CartanRestriction} therefore identifies
  the obstruction with $\tau(c_G)$. This is exactly the rank-one proof
  with $A_1$ replaced by the simply laced root lattice.

  On $M^{[k]}=\A_{G_1}(I)^{\bar\otimes k}$ take the diagonal tensor
  lifts $\alpha_g^{[k]}=\alpha_{g,1}^{\otimes k}$ and defects
  $u^{[k]}(g,h)=u_1(g,h)^{\otimes k}$. Their obstruction is
  $k[\omega_{\theta_1}]$ by Proposition~\ref{prop:DiagonalAdditivity}.
  They remain faithful: if $\alpha_{g,1}^{\otimes k}=\Ad W$ with
  $W\in\U(M^{[k]})$, a nonzero normal slice of $W$ over the last $k-1$
  factors gives an intertwiner $a\in\A_{G_1}(I)$ with
  $\alpha_{g,1}(x)a=ax$ for all $x$. Factoriality makes $a^*a$ and
  $aa^*$ scalar, so its polar part is a unitary implementing
  $\alpha_{g,1}$. Level-one faithfulness forces $g=e$.
  The finite tensor product $M^{[k]}$ is again hyperfinite of type
  \threeone and is isomorphic to $R$.

  For a semisimple product with simply laced factors use
  $\bar\otimes_j M_j^{[k_j]}$ and component tensor lifts. The same
  slicing argument proves faithfulness, and the associator is the
  product of the component scalar associators, giving the tuple of
  levels. This proves the introduction's realization statement without
  asserting faithfulness on the separate level-$k$ vacuum factor.
\end{proof}

\begin{rmk}[Non-simply-laced vs.\ non-simply-connected]
  \label{rmk:NonSimplyLaced}
  For every compact connected $G$, Theorem~\ref{thm:CartanRestriction}
  embeds $H^4(BG,\ZZ)$ in the torsion-free group $H^4(BT,\ZZ)$.
  Thus non-simply-connectedness does not introduce torsion in degree four.
  For a finite central quotient $\widetilde G\to G$, admissible levels on
  the cover can nevertheless be restricted by integrality on the larger
  cocharacter lattice. Multiplicative-gerbe descent describes such constraints
  \cite{GaWa2009}; ordinary $H^3(G,\ZZ)$ must not be confused with
  $H^4(BG,\ZZ)$. Moreover $\pi_1(G)$ is the subgroup used for the central
  quotient, not necessarily the entire center of $\widetilde G$.

  For non-simply-laced simply connected groups, $\Qvee$ is still the
  appropriate even lattice. This supplies a torus quadratic form, but not
  equality of the nonabelian level-one net with a lattice net. A compatible
  torus comparison on the ambient factor is still required to apply
  Theorem~\ref{thm:main} in types $B_n,C_n,F_4,G_2$;
  Theorem~\ref{thm:AnomalyAllTypes} avoids it.
  At level $k$ the same torus form is $(\slot,\slot)_k$;
  equivalently one may use the Euclidean lattice
  $L_k=\sqrt{k}\,\Qvee$ and the torus isomorphism
  \[
    F/\Qvee\longrightarrow F/L_k,\qquad
    x+\Qvee\longmapsto\sqrt{k}x+L_k.
  \]
  Pullback of $y\mapsto\basicform{y}{y}/2$ is precisely
  $x\mapsto k\basicform{x}{x}/2$. This identifies the lattice at
  every level in every type. For the proof it is enough to make the
  representation/defect comparison at level one and use
  Proposition~\ref{prop:DiagonalAdditivity} for higher levels.
\end{rmk}

\subsection{Every class on the hyperfinite type \texorpdfstring{III$_1$}{III1} factor}
\label{ssec:EveryClass}

The only input needed is level one. Once the generator $c_G$ is
realized for every simple type, tensor products and opposite factors
produce every class, and products of groups are handled factorwise.
Level one in the simply laced types is
Theorem~\ref{thm:SimplyLacedRealization}; the remaining simple types are
reached by restriction (Remark~\ref{rmk:IndexOne}).

\begin{lem}[Pullback and opposite kernels]
  \label{lem:PullbackOpposite}
  Let $\theta\colon H\to\Out(M)$ be an $H$-kernel on a factor $M$ with
  Sutherland obstruction $[\omega]$.
  \begin{enumerate}
    \item If $\varphi\colon G\to H$ is an injective continuous
      homomorphism, then $\theta\circ\varphi$ is a $G$-kernel with
      obstruction $\varphi^*[\omega]$.
    \item Put $\alpha^{\op}(x^{\op})=\alpha(x)^{\op}$ on $M^{\op}$. Then
      $\theta^{\op}(h)=[\alpha_h^{\op}]$ is an $H$-kernel on $M^{\op}$
      with obstruction $-[\omega]$.
  \end{enumerate}
\end{lem}
\begin{proof}
  (1) Composing injective homomorphisms is injective, and
  $\alpha_{\varphi(g)}$, $u(\varphi(g),\varphi(h))$ are Borel lifts and
  defects for $\theta\circ\varphi$ whose associator is
  $\omega\circ\varphi^{\times3}$.

  (2) $\alpha^{\op}$ is inner in $M^{\op}$ exactly when $\alpha$ is inner
  in $M$, since
  $(\Ad u)^{\op}=\Ad\big((u^*)^{\op}\big)$; so $\theta^{\op}$ is
  injective. From $\alpha_g\alpha_h=\Ad u(g,h)\,\alpha_{gh}$ we get
  $\alpha_g^{\op}\alpha_h^{\op}=\Ad u'(g,h)\,\alpha_{gh}^{\op}$ with
  $u'(g,h)=(u(g,h)^*)^{\op}$. Applying $x\mapsto(x^*)^{\op}$, which
  reverses products and conjugates scalars, to the associator identity
  for $u$ gives
  \[
    \alpha_g^{\op}(u'(h,l))\,u'(g,hl)=\overline{\omega(g,h,l)}\,
    u'(g,h)\,u'(gh,l).
  \]
\end{proof}

\begin{lem}[Realizable classes form a group]
  \label{lem:RealizableSubgroup}
  Let $G$ be a compact group and let $\mathcal R(G)\subset
  H^3_{\Borel}(G,\TT)$ be the set of Sutherland obstructions of
  $G$-kernels on $R$. If $\mathcal R(G)$ is nonempty, it is a subgroup.
\end{lem}
\begin{proof}
  For $G$-kernels $\theta,\theta'$ on $R$, the diagonal tensor product
  $g\mapsto[\alpha_g\otimes\alpha'_g]$ on $R\bar\otimes R\cong R$ with
  defects $u\otimes u'$ has associator $\omega\omega'$. It is faithful:
  if $\alpha_g\otimes\alpha'_g=\Ad W$, a nonzero normal slice of $W$
  over the second factor intertwines $\alpha_g$ with the identity, and
  factoriality makes its polar part a unitary implementing $\alpha_g$,
  so $g=e$. Hence $\mathcal R(G)$ is closed under addition. It is closed
  under negation by Lemma~\ref{lem:PullbackOpposite}(2), since
  $R^{\op}$ is an injective factor of type \threeone and hence
  $R^{\op}\cong R$ \cite{Ha1987}. A nonempty subset closed under
  addition and negation is a subgroup; in particular
  $\theta\otimes\theta^{\op}$ realizes $0$.
\end{proof}

\begin{rmk}[Index-one embeddings]
  \label{rmk:IndexOne}
  For $\varphi\colon G\to H$ between simple simply connected groups,
  $\varphi^*c_H=j_\varphi c_G$ with $j_\varphi$ the Dynkin embedding
  index. The non-simply-laced types embed with index one into simply
  laced groups \cite{Dy1952}:
  \[
    \Sp(n)\subset\SU(2n),\quad
    \Spin(2n+1)\subset\Spin(2n+2),\quad
    G_2\subset\Spin(8),\quad
    F_4\subset E_6 .
  \]
  By Theorem~\ref{thm:CartanRestriction},
  $j_\varphi=\tfrac12\langle\varphi\alpha^\vee,\varphi\alpha^\vee\rangle_H$
  for a long root $\alpha$ of $G$, so index one says that a long-root
  coroot of $G$ is a coroot of $H$. This is checked directly: for
  $\Sp(n)\subset\SU(2n)$ the coroot $e_i$ of $2\varepsilon_i$ maps to
  $e_i-e_{n+i}$; the long roots of $B_n$ are roots of $D_{n+1}$; and the
  long roots of $G_2$ and $F_4$ span root subgroups $\SU(3)\subset\Spin(8)$
  and $\Spin(8)\subset E_6$.
\end{rmk}

\begin{cor}[Every obstruction on the hyperfinite type \threeone factor]
  \label{cor:EveryObstruction}
  Let $G$ be a compact, connected, simply connected, simple Lie group.
  Every class in $H^3_{\Borel}(G,\TT)\cong H^4(BG,\ZZ)$ is the Sutherland
  obstruction of a $G$-kernel $G\to\Out(R)$ on the hyperfinite type
  \threeone factor $R$.
\end{cor}
\begin{proof}
  It suffices to realize the generator $c_G$ at level one: then
  $\mathcal R(G)$ is a subgroup of $H^4(BG,\ZZ)=\ZZ c_G$ containing
  $c_G$ (Lemma~\ref{lem:RealizableSubgroup}). For simply laced $G$ the
  level-one case of Theorem~\ref{thm:SimplyLacedRealization} realizes
  $c_G$. For the other types restrict the level-one kernel of $H$ along
  an embedding of Remark~\ref{rmk:IndexOne}; by
  Lemma~\ref{lem:PullbackOpposite}(1) and naturality of
  \eqref{eq:Transgression} its class is $\varphi^*c_H=c_G$.
\end{proof}

\begin{rmk}[Semisimple groups]
  \label{rmk:EveryObstructionSemisimple}
  Corollary~\ref{cor:EveryObstruction} already gives the semisimple case.
  For $G=\prod_jG_j$ simply connected with simple factors, each $BG_j$ is
  $3$-connected with $H^4(BG_j,\ZZ)\cong\ZZ$, so the Künneth formula
  gives $H^4(BG,\ZZ)=\bigoplus_j\ZZ c_{G_j}$ with no mixed or torsion
  terms. A class $(k_j)_j$ is realized by $\bar\otimes_j\theta_j$ on
  $\bar\otimes_jR\cong R$, where $\theta_j$ realizes $k_jc_{G_j}$: the
  associator is $\prod_j\omega_j$, and faithfulness follows by the
  slicing argument of Lemma~\ref{lem:RealizableSubgroup}, as in the proof
  of Theorem~\ref{thm:main}. Thus every class in $H^4(BG,\ZZ)$ is
  realized on $R$ for every compact simply connected semisimple $G$.
\end{rmk}

\subsection{The jump-loop kernel in every type}
\label{ssec:JumpLoopAllTypes}

The restriction used in Corollary~\ref{cor:EveryObstruction} can be
carried out on the loop group nets themselves. The argument is that of
Proposition~\ref{prop:DiagonalAdditivity}, with the diagonal
$G\to G^k$, of index $k$, replaced by an index-one embedding. Below,
$[\omega_{\theta}]$ denotes the associator class of a Borel lift of a
homomorphism $\theta\colon G\to\Out(M)$ into the outer automorphisms
of a factor, as in Section~\ref{ssec:Gkernels}. It is independent of the
lift whether or not $\theta$ is injective, and it vanishes for the
trivial homomorphism.

\begin{prop}[Restriction along index-one embeddings]
  \label{prop:IndexOneRestriction}
  Let $\varphi\colon G\to H$ be an injective homomorphism of compact
  simply connected simple groups with $\varphi^*c_H=c_G$, and let
  $k\ge1$. For the jump-loop homomorphisms $\theta^G_k$ and $\theta^H_k$
  of Proposition~\ref{prop:JumpHomomorphism},
  \[
    [\omega_{\theta^G_k}]\ =\ \varphi^*[\omega_{\theta^H_k}].
  \]
\end{prop}
\begin{proof}
  The basic forms satisfy $\varphi^*\langle\slot,\slot\rangle_H=
  \langle\slot,\slot\rangle_G$, since both are invariant forms on the
  simple Lie algebra of $G$ and they agree on a long coroot. Hence the
  affine cocycle $\kappa_k$ of $H$ pulls back along $\LL\varphi$ to that
  of $G$, and the level-$k$ extension of $\LL H$ pulls back to the
  level-$k$ extension of $\LL G$ \cite{PS1986}*{Chapter~4}.

  The vacuum vector $\Omega$ of $\Hil_{H_k}$ is annihilated by the
  positive modes of the affine algebra of $H$ and is fixed by $H$, so it
  is a highest-weight vector of weight $k\Lambda_0$ for the affine algebra
  of $G$. As in the proof of Proposition~\ref{prop:DiagonalAdditivity},
  the closed $\LL G$-invariant subspace $\Hil'\subset\Hil_{H_k}$ it
  generates carries the level-$k$ vacuum representation of $\LL G$
  \cite{PS1986}*{Chapter~9}.

  Choose profiles $\gamma_g$ for $\theta^G_k$ and use $\varphi\circ\gamma_g$
  as profiles for $\theta^H_k$ at $\varphi(g)$; this is allowed by
  Proposition~\ref{prop:JumpHomomorphism}, and it gives a Borel lift of
  $\theta^H_k\circ\varphi$. The operators
  \[
    \pi_{H,k}(\varphi\circ\gamma_g),\qquad
    \pi_{H,k}(\varphi\circ\eta_I(g,h))
  \]
  represent actual loop-group elements of $\LL G$, so they preserve
  $\Hil'$. Their restrictions are the level-$k$ lift implementers and
  local defects of $\theta^G_k$, up to the scalar choices of section.
  Restricting the associator identity of $\theta^H_k\circ\varphi$ to
  $\Hil'$ gives the associator identity of $\theta^G_k$ with scalar
  $\omega_{\theta^H_k}\circ\varphi^{\times3}$. Section changes contribute
  coboundaries.
\end{proof}

\begin{lem}[Normal subgroups of simple compact groups]
  \label{lem:NormalSubgroups}
  Let $G$ be a compact connected Lie group with simple Lie algebra. Every
  normal subgroup $N\subset G$, closed or not, is either contained in
  $Z(G)$ or equal to $G$.
\end{lem}
\begin{proof}
  Let $g\in N\setminus Z(G)$ and $c(x)=xgx^{-1}g^{-1}$. The image of
  $c$ is a path-connected subset of $N$ containing $e$ and not equal to
  $\{e\}$. Since $x\,c(y)\,x^{-1}=c(xy)\,c(x)^{-1}$, the subgroup $S$
  generated by this image is normal; it is path connected. By Yamabe's
  theorem $S$ is an analytic subgroup \cite{Ya1950}, so its Lie algebra is
  a nonzero ideal, hence everything, and $S=G$. Thus $N=G$.
\end{proof}

\begin{prop}[Central elements act by outer automorphisms]
  \label{prop:CentreOuter}
  Let $G$ be a compact simply connected simple Lie group, $k\ge1$, and
  $z\in Z(G)\setminus\{e\}$. Then $\alpha_{z,k}$ is outer on
  $M_k=\A_{G_k}(I)$.
\end{prop}
\begin{proof}
  Write $\pi=\pi_k$ and $\gamma=\gamma_z$, so $\gamma=e$ on $(c,a)$ and
  $\gamma=z$ on $(b,d)$. Let $\zeta\colon\RR\to G$ equal $\gamma$ on
  $[c,d]$ and $z$ on $[d,c+2\pi]$, extended by $\zeta(x+2\pi)=z\,\zeta(x)$.
  It is smooth, since it is locally constant near $c$ and $d$, so
  $\zeta$ is a discontinuous loop in $\LL_{Z(G)}G$ representing the
  nontrivial class of $z$ in $\LL_{Z(G)}G/\LL G\cong Z(G)$
  \cite{TL1999}*{Section~4.3}. Conjugation by $\zeta$ is an automorphism
  of $\LL G$ because $z$ is central.

  Suppose $\alpha_{z,k}=\Ad u$ on $M_k$ with $u\in\U(M_k)$. Put
  $J=(c,d)$ and $K=I'$, so $J\cup K=\SOne$. On $\A(c,a)$ and $\A(b,d)$
  both $\Ad u$ (by locality) and $\Ad\pi(\gamma)$ (by
  Lemma~\ref{lem:FlankTriviality}, $z$ being central) are trivial. By
  strong additivity, which is part of complete rationality
  \cite{CaKaLoWi2015}*{Section~3.3} and hence holds by
  Corollary~\ref{cor:LoopGroupNetAnalytic},
  $\A(J)=\A(c,a)\vee M_k\vee\A(b,d)$, so $\Ad u=\Ad\pi(\gamma)$ on
  $\A(J)$. Since $\zeta=\gamma$ on $J$, in $\PU(\Hil_{G_k})$
  \[
    u\,\pi(\ell)\,u^*=\pi(\zeta\ell\zeta^{-1}),\qquad \ell\in\LL_JG.
  \]
  On $K$ the values of $\zeta$ are central, so $\zeta\ell\zeta^{-1}=\ell$
  for $\ell\in\LL_KG$, and $u$ commutes with $\pi(\LL_KG)$ by locality.
  Hence the same identity holds for $\ell\in\LL_KG$.

  The group $\LL G$ is generated by $\LL_JG\cup\LL_KG$: it is connected
  because $G$ is simply connected, and a loop $\exp\xi$ near $e$ factors
  as $\exp(f\xi)\cdot\exp(f\xi)^{-1}\exp\xi$ with $f$ a cutoff supported
  in $J$ and equal to $1$ outside $K$. Therefore $u$ intertwines $\pi$
  with the conjugated representation $\ell\mapsto\pi(\zeta\ell\zeta^{-1})$.
  By \cite{TL1999}*{Theorem~4.3.3 and Proposition~4.1.4} the latter is
  the irreducible positive-energy representation of level $k$ whose
  highest weight is the image of $0$ under the action of
  $z^{-1}\ne e$ on the level-$k$ alcove, namely $k\lambda^\vee_i\neq0$
  for the fundamental coweight $\lambda^\vee_i$ of the special node
  attached to $z^{-1}$. Irreducible positive-energy representations of
  the same level with different highest weights are inequivalent, a
  contradiction.
\end{proof}

\begin{thm}[Anomaly is the level in every type]
  \label{thm:AnomalyAllTypes}
  Let $G$ be a compact simply connected simple Lie group and $k\ge1$.
  The jump-loop homomorphism
  $\theta_k\colon G\to\Out(\A_{G_k}(I))$ satisfies
  \[
    [\omega_{\theta_k}]\ =\ \tau(kc_G).
  \]
  Moreover $\theta_k$ is injective, so it is a $G$-kernel on the
  hyperfinite type \threeone factor $\A_{G_k}(I)$.
\end{thm}
\begin{proof}
  For simply laced $G$, Corollary~\ref{cor:SimplyLaced} gives
  $[\omega_{\theta_1}]=\tau(c_G)$, and
  Proposition~\ref{prop:DiagonalAdditivity} gives
  $[\omega_{\theta_k}]=\tau(kc_G)$. Otherwise choose
  $\varphi\colon G\to H$ from Remark~\ref{rmk:IndexOne} with $H$ simply
  laced. By Proposition~\ref{prop:IndexOneRestriction}, the simply laced
  case for $H$, and naturality of~\eqref{eq:Transgression},
  $[\omega_{\theta_k}]=\varphi^*\tau(kc_H)=\tau(kc_G)$.

  The kernel $N$ of $\theta_k$ is normal. It is not $G$, because the
  trivial homomorphism has class $0\ne\tau(kc_G)$ in the torsion-free
  group $H^4(BG,\ZZ)$. By Lemma~\ref{lem:NormalSubgroups},
  $N\subset Z(G)$, and $N=\{e\}$ by Proposition~\ref{prop:CentreOuter}.
  The factor is hyperfinite of type \threeone by
  Corollary~\ref{cor:LoopGroupNetAnalytic}.
\end{proof}

\begin{rmk}
  \label{rmk:EveryObstructionScope}
  By Theorem~\ref{thm:AnomalyAllTypes} every positive class is realized
  by the jump-loop kernel $\theta_k$ on $\A_{G_k}(I)\cong R$ itself. The
  negative and zero classes in Corollary~\ref{cor:EveryObstruction}
  still come from opposite factors and tensor products.
\end{rmk}

\section{Anomaly is the level}
\label{sec:AnomalyIsLevel}

By Theorem~\ref{thm:AnomalyAllTypes}, for every compact simply connected
simple $G$ and $k\ge1$,
\begin{equation}
  \label{eq:AnomalyIsLevel}
  \tau^{-1}([\omega_{\theta_k}])
  \ =\ k\,c_{\mathrm{basic}}\ \in H^4(BG,\ZZ).
\end{equation}
Here $c_{\mathrm{basic}}$ is normalized by its torus restriction. Three
different cocycles enter the construction: a scalar group two-cocycle
$c_{\LL G}$ describing a chosen section of the loop-group central extension,
the factor-valued defect $u_k(g,h)$ of
Proposition~\ref{prop:JumpHomomorphism}, and its scalar three-associator
$\omega_{\theta_k}$. Factoriality makes the last object scalar; it does
not identify it with $c_{\LL G}$ without the comparison argument.

The differentiated central extension has affine Lie algebra cocycle
\[
  \kappa_k(\xi,\eta)\ =\ \frac{1}{2\pi}\!\int_{0}^{2\pi}\!
    (\xi,\eta')_k\,\dd\theta,
\]
with the basic form and circle-coordinate convention fixed in
Section~\ref{sec:Notation}~\cite{PS1986}. Geometric transgression relates
the basic characteristic class to an integral class in $H^3(G,\ZZ)$,
represented by a suitably normalized invariant Wess--Zumino three-form.
This geometric description is motivation. The operative proof is
level-one torus/lattice associator comparison, followed by integral Cartan
restriction (the Weyl-invariant quadratic form determines the simple-group
class), followed by level additivity. In this construction, projectively
equivalent loop-group representations give the same scalar calculation
when their lifts and local defects are identified. Merely knowing the
central-extension class, without this identification, is not the comparison
proof.

\begin{rmk}[Remaining existence problem]
  \label{rmk:RemainingExistence}
  The scalar tensor calculation in Proposition~\ref{prop:DiagonalAdditivity}
  is independent of the normalization of the base class. The latter is
  supplied by the lattice computation together with Cartan restriction,
  not by additivity alone. Outside types $A$, $D$, $E$ the class is
  obtained by restriction (Theorem~\ref{thm:AnomalyAllTypes}), so the
  ambient-factor comparison of Proposition~\ref{prop:ReduceCartan} is not
  needed for it. The analytic net properties are supplied by
  Corollary~\ref{cor:LoopGroupNetAnalytic}, and faithfulness by
  Lemma~\ref{lem:NormalSubgroups} and Proposition~\ref{prop:CentreOuter}.
  The ambient comparison remains open only as a statement about the torus
  route of Theorem~\ref{thm:main}.
\end{rmk}

\def\cprime{$'$}\newcommand{\noopsort}[1]{}

\end{document}